\documentclass[12pt,a4paper]{amsart}
\usepackage{mathrsfs}
\usepackage{amsfonts}
\usepackage{txfonts}

\usepackage{hyperref}
\usepackage{latexsym}
\usepackage{amssymb}
\usepackage{geometry}
\newtheorem{theorem}{Theorem}[section]
\newtheorem{lemma}[theorem]{Lemma}
\newtheorem{corollary}[theorem]{Corollary}
\newtheorem{proposition}[theorem]{Proposition}
\newtheorem{example}[theorem]{Example}

\theoremstyle{definition}
\newtheorem{definition}[theorem]{Definition}

\newtheorem{remark}[theorem]{Remark}

\numberwithin{equation}{section}

\begin{document}

\title[On the sum of simultaneously  proximinal sets  ]
{{\bf On the sum of simultaneously proximinal sets}}

\author{Longfa Sun }

\address{ Longfa Sun:  School of Mathematics and Physics, North China Electric Power University,
 Baoding 071003,    P.R. China}

\email{sun.longfa@163.com\;\;(L. Sun)}

\author{Yuqi Sun}

\address{ Yuqi Sun:  School of Mathematical Science, Xiamen University,
 Xiamen 361005,    P.R. China}

\email{sunyuqisyq@qq.com\;\;(Y. Sun)}

\author{Wen Zhang}

\address{ Wen Zhang:  School of Mathematical Science, Xiamen University,
 Xiamen 361005,    P.R. China}

\email{wenzhang@xmu.edu.cn\;\;(W. Zhang) Corresponding author}

\author{Zheming Zheng}

\address{ Zheming Zheng:  School of Mathematical Science, Xiamen University,
 Xiamen 361005,    P.R. China}

\email{3184532127@qq.com\;\;(Z. Zheng)}

\thanks{Longfa Sun's research was supported by the Fundamental Research Funds for the Central Universities 2019MS121}
\thanks {Wen Zhang's research was supported by the National Natural Science Foundation of China, No: 11731010.}

\date{}

\begin{abstract}
In this paper, we show that the sum of a compact convex subset and a simultaneously $\tau$-strongly proximinal convex subset (resp. simultaneously approximatively $\tau$-compact convex subset) of a Banach space X is simultaneously $\tau$-strongly proximinal (resp. simultaneously approximatively $\tau$-compact ), and the sum of weakly compact convex subset and a simultaneously approximatively weakly compact convex subset of X is still  simultaneously approximatively weakly compact, where $\tau$ is the norm or the weak topology. Moreover, some related results on the sum of simultaneously proximinal subspaces are presented.

\end{abstract}



\maketitle

\section{Introduction}
Let $X$ be a real Banach space and $C$ a nonempty closed subset of $X$.  An element $y_0\in C$ is called a best approximation to $x$ from $C$ if
\[\|x-y_0\|=\inf_{y\in C}\|x-y\|\equiv d(x,C).\] For any $x\in X$,
let $P_C(x)=\{y\in C: \|x-y\|=d(x,C)\}$. $C$ is said to be $proximinal$ if the set $P_C(x)$ is nonempty for every $x$ in $X$.

For a given bounded subset $A$, all elements of $A$ might be approximated simultaneously by a single element of $C$. This type of problem arises when a function being approximated is not known exactly but is known to belong to a set\cite{{Ma},{S}}.

For a bounded subset $A\subset X$, an element $y_0\in C$ is called a best simultaneous approximation of $A$ from $C$  if
\[\sup\limits_{a\in A}\|a-y_0\|=\inf\limits_{y\in C}\sup\limits_{a\in A}\|a-y\|\equiv d(A,C).\]
For any bounded subset $A\subset X$, let $P_C(A)=\{y\in C:\sup\limits_{a\in A}\|a-y\|=d(A,C)\}$. $C$ is said to be simultaneously proximinal if $P_C(A)$ is nonempty for  every bounded subset $A\subset X$ \cite{N}.
For $\delta>0$, let $P_C(A,\delta)=\{y\in C:\sup\limits_{a\in A}\|a-y\|<d(A,C)+\delta\}$. A sequence $\{y_n\}\subset C$ is called $minimizing$ for $A\subset X$ if $\sup\limits_ {a\in A}\|a-y_n\|\rightarrow d(A,C)$.

Taking the set $A$ to be a singleton, it follows that simultaneously proximinal sets are proximinal. It is known that any weakly compact subsets or any reflexive subspaces of $X$ are simultaneously proximinal\cite{{M},{R}}.

The notions of simultaneous approximative compactness and simultaneous strong proximinality were introduced by Gupta and Narang \cite{G}. We extend those as following. In this paper, unless otherwise mentioned, we denote by $\tau$ either the norm or the weak topology on $X$. As usual, in case $\tau$ is the norm topology, we omit it.

\begin{definition}
A subset $C$ of $X$ is said to be simultaneously approximatively $\tau$-compact if for each bounded set $A\subset X$, every minimizing sequence $\{y_n\}\subset C$ for $A$ has a subsequence $\tau$-convergent to an element in $C$.
\end{definition}
\begin{definition}
A $\tau$-closed set $C$ of $X$ is said to be simultaneously $\tau$-strongly proximinal if $C$ is simultaneously   proximinal for each bounded set $A\subset X$ and for any $\tau$-neighbourhood $V$ of $0$ in $X$, there exists $\delta>0$ such that $P_C(A,\delta)\subset P_C(A)+V$.
\end{definition}

The following question was raised by Cheney and Wulbert in \cite{C}:

If $F$ and $G$ are proximinal subspaces of a Banach space $X$, and $F+G$ is closed, does it follow that $F+G$ is proximinal in $X$?

In\cite{F}, Feder gave a negative answer to this problem and proved that if $F$ is reflexive and $G$ is proximinal such that $F+G$ is closed and $F\cap G$ is finite dimensional, then $F+G$ is proximinal. Lin\cite{L} , Deeb and Khalil \cite{D} proved that the condition  ``$F\cap G$ is finite dimensional'' can be dropped and the above conclusion still established.

Rawashdeh, Al-Sharif and Domi \cite{R} generalized the Feder's result to the sum of simultaneously proximinal subspaces and proved
if $F$ is reflexive and $G$ is simultaneously proximinal satisfying $F+G$ is closed such that $F\cap G$ is finite dimensional then $F+G$ is simultaneously proximinal. Meng, Luo, and Shi \cite{M} proved that a weakly compact convex subset and a simultaneously proximianl convex subset is simultaneously proximinal. This can be regarded as a localized version as the above conclusions. Furthermore, Gupta and Narang \cite{G} generalized  Rawashdeh's result to the sum of simultaneously strongly proximinal subspaces. For the recent development of this topic, we refer to \cite{{rao1},{rao2}} and references therein.

In this paper, we shall study the sum of simultaneously proximinal subsets of a Banach space $X$. We prove that the sum of a compact convex subset and a simultaneously $\tau$-strongly proximinal convex subset (resp. simultaneously approximatively $\tau$-compact convex subset) of $X$ is simultaneously $\tau$-strongly proximinal (resp. simultaneously approximatively $\tau$-compact ), and the sum of weakly compact convex subset and a simultaneously approximatively weakly compact convex subset of $X$ is simultaneously approximatively weakly compact. As an application, some related results on the sum of simultaneously proximinal subspaces are presented.

All symbols and notations in this paper are standard. We use $X$ to denote a real Banach space. $B_X$ (resp. $S_X$) stands for the closed unit ball (resp. the unit sphere) of $X$. For a subset $A\subset X$,  ${\rm co}(A)$ denotes the convex hull of $A$.

\section{General results}
In this section, we consider the simultaneous proximinality of $\tau$-closed subsets in a Banach space. It was shown that a reflexive subspace or a weakly compact subset of a Banach space $X$ is simultaneously proximinal \cite{{M},{R}}, and a finite dimensional subspace is simultaneously approximatively compact \cite{G}. We will show that every $\tau$-compact subsets are $\tau$-simultaneously approximatively compact, and we shall characterize reflexive spaces from simultaneous proximinality point of view.

\begin{proposition}
Let $C$ be a nonempty set of a Banach space $X$. Then\\
(1) $C$ is simultaneously proximinal if and only if $z+C$ is simultaneously proximinal for any $z\in X$, and if and only if $\lambda C$ is simultaneously proximinal for any $\lambda>0$.\\
(2) $C$ is simultaneously $\tau$-strongly proximinal if and only if $z+C$ is simultaneously $\tau$-strongly proximinal for any $z\in X$, and if and only if $\lambda C$ is simultaneously $\tau$-strongly proximinal for any $\lambda>0$.\\
(3) $C$ is simultaneously approximatively $\tau$-compact if and only if $z+C$ is simultaneously approximatively $\tau$-compact for any $z\in X$, and if and only if $\lambda C$ is simultaneously approximatively $\tau$-compact for any $\lambda>0$.
\end{proposition}
\begin{proof}
The proof is elementary. It is sufficient to note that for every bounded set $A\subset X$,
\[d(A,C)=d(z+A,z+C), d(\lambda A,\lambda C)=\lambda d(A,C)\]
and
\[P_{z+C}(z+A)=z+P_C(A), P_{\lambda C}(\lambda A)=\lambda P_C(A).\]
\end{proof}

\begin{proposition}
Suppose that $C$ is a $\tau$-compact set of a Banach space $X$. Then $C$ is simultaneously approximatively $\tau$-compact.
\end{proposition}
\begin{proof}
Let $A\subset X$ be a bounded subset of $X$ and $\{y_n\}\subset C$ a minimizing sequence for $A$, i.e.
\[\lim\limits_{n\rightarrow\infty}\sup\limits_{a\in A}\|a-y_n\|=d(A,C).\]
By the $\tau$-compactness of $C$, $\{y_n\}$ has a $\tau$-convergent subsequence. The proof is complete.
\end{proof}

\begin{corollary}
Let $E$ be a subspace of a Banach space $X$. Then\\
(1) if $E$ is reflexive, then $E$ is simultaneously approximatively weakly compact.\\
(2) if $E$ is finite dimensional, then $E$ is simultaneously approximatively compact.
\end{corollary}

\begin{proof}
(1) Suppose that $A\subset X$ be a bounded subset and $\{y_n\}\subset E$ a minimizing sequence for $A$, then $\{y_n\}$ is bounded. Let $\lambda=\sup\limits_{n}\|y_n\|$, then
\[d(A,E)=d(A,\lambda B_E)\;\; \text{and}\;\;\; \{y_n\}\subset \lambda B_E.\]
Since $B_E$ is weakly compact, it follows from  Proposition 2.1 and Proposition 2.2 that $\lambda B_E$ is simultaneously approximatively weakly compact. This implies that  $\{y_n\}$ has a weakly convergent subsequence, and so $E$ is  simultaneously approximatively weakly compact.

(2) The proof is similar to (1), it is sufficient to substitute compactness for weak compactness.
\end{proof}
Gupta and Narang\cite{G}  showed that a closed subset $C$ of $X$ is simultaneously approximatively compact if and only if $C$ is simultaneously strongly proximinal and $P_C(A)$ is compact for every bounded subset $A$ of $X$. When $C$ is weakly closed, we have the following result.

\begin{theorem}
Let $C$ be a weak closed subset of a Banach space $X$ and $A\subset X$ be a bounded subset. If $C$ is simultaneously approximatively weakly compact for $A$, then $C$ is simultaneously weakly strongly proximinal  for  $A$ and $P_C(A)$ is weakly compact.
\end{theorem}
\begin{proof}
Note that in weak topology, weak compactness and weakly sequential compactness coincide. It follows that if $C$ is simultaneously approximatively weakly compact , then $P_C(A)$ is weakly compact, for every bounded subset $A\subset X$.

If $C$ is not simultaneously weakly strongly proximinal, then there exists a bounded subset $A\subset X$,  a weak neighbourhood of $0$ and a minimizing sequence $\{y_n\}\subset C$ for $A$ with $y_n\notin P_C(A)+V$. Since $C$ is simultaneously approximatively weakly compact, $\{y_n\}$ has a weakly convergent subsequence $\{y_{n_k}\}$ with $y_{n_k} \stackrel{w}{\rightarrow} y_0$.
By the weakly lower semi-continuity of the norm, we have $y_0\in P_C(A)$. Therefore, there exist some $n\geq1$ such that
$y_n\in y_0+V\subset P_C(A)+V$. A contradiction!

\end{proof}

Note that if $C$ is a closed convex set, then $C$ is simultaneously approximatively compact$\Rightarrow$ $C$ is simultaneously approximatively weakly compact and $C$ is simultaneously strongly proximinal $\Rightarrow$ $C$ is simultaneously weakly strongly proximinal. The following example due to Dutta\cite{DS} will show that none of the implications can be reversed.

\begin{example}\cite[Example 2.3]{DS}
Consider the sequence $\{x_n\}$ in $c_0$  where $x_n=(-\frac{1}{n},0,\cdots,1,0,\cdots)$, where $1$ occurs at the nth place. Note that $x_n\rightarrow 0$ weakly. Let $C=\overline{co}\{x_n:n\in\mathbb{N} \}$. Then $C$ is weakly compact
hence simultaneously approximatively weakly compact. Since $C$ is convex, by Theorem 2.4, $C$ is simultaneously weakly strongly proximinal. It was shown that $C$ is not  strongly proximinal in \cite{DS}. Further, $C$ is not  simultaneously strongly proximinal.
\end{example}

 The following result characterizes reflexive spaces from simultaneous proximinality point of view.
\begin{theorem}
Let $X$ be a Banach space. Then the following statements are equivalent.\\
(1) $X$ is reflexive.\\
(2) Every closed convex set is simultaneously approximatively weakly compact.\\
(3) Every closed convex set is simultaneously weakly strongly proximinal.\\
(4) Every closed convex set is simultaneously proximinal.
\end{theorem}
\begin{proof}
(1)$\Rightarrow$(2): Let $C$ be a closed convex set of $X$ and $A\subset X$ be a bounded subset. suppose $\{y_n\}\subset C$ is a minimizing sequence for $A$, then $\{y_n\}$ is bounded. Since $X$ is reflexive, $\{y_n\}$ is relatively weakly compact. Thus,  $\{y_n\}$ has a weakly convergent subsequence.

(2)$\Rightarrow$(3): Follows from Theorem 2.4.

(3)$\Rightarrow$(4): Obviously.

(4)$\Rightarrow$(1): Note that a simultaneously proximinal set is proximinal, this follows from \cite[Theorem 2.8]{B}.
\end{proof}

\section{Sum of simultaneously proximinal sets}
In this section, we discuss the simultaneous proximinality under sum operation. Firstly, we will show that there exist two simultaneously approximatively compact sets satisfying the sum is closed but not simultaneously proximinal in any infinite-dimensional  Banach space $X$. The following lemma is classical.
\begin{lemma}
Let $X$ be a Banach space and $Y$  a proper closed subspace of $X$. Then for every $0<\varepsilon<1$, there exists
$x\in S_X$ such that $d(x,Y)>\varepsilon$.
\end{lemma}

By the Lemma 3.1, we have

\begin{lemma}
Let $X$ be a Banach space and $Y$  a proper closed subspace of $X$. Then for every $0<\varepsilon<1$, there exists
$x\in X$ with $\|x\|=\frac{1}{\varepsilon}$ such that $d(x,Y)>1$.
\end{lemma}
\begin{proof}
It is sufficient to note that $d(\frac{1}{\varepsilon}x,Y)=\frac{1}{\varepsilon}d(x,Y).$
\end{proof}

The following result is motivated by Pyatyshev's construction\cite{P}.

\begin{theorem}
Let $X$ be a infinite-dimensional Banach space. Then there exist two simultaneously approximatively compact sets satisfying the sum is closed but not simultaneously proximinal.
\end{theorem}
\begin{proof}
Let $\{\varepsilon_n\}$ be a number sequence satisfying: $\frac{1}{2}<\varepsilon_n<1$ and $\varepsilon_n\rightarrow 1$.
Let $x_0\in S_X$, by Lemma 3.2, we can choose a sequence $\{x_n\}\subset X$ such that for any $n\in\mathbb{N}$,
\[\|x_n\|=\frac{1}{\varepsilon_n},\;\;\;\;d(x_n,span\{x_0,x_1,\cdots,x_{n-1}\})>1.\]
Therefore, $1<\|x_n\|<2$ and $\|x_n\|\rightarrow 1$.

For any $n\in\mathbb{N}$, we can choose $\lambda_n>0$, such that $\|\lambda_nx_0+x_n\|=n$. Introduce the sets
\[A=span\{x_0\},\;\;\;\;B=\bigcup^\infty_{n=1}\{\lambda_nx_0+x_n\}.\]
Then
\[A+B=span\{x_0\}+\bigcup^\infty_{n=1}\{\lambda_nx_0+x_n\}=\bigcup^\infty_{n=1}\{span\{x_0\}+x_n\}.\]

Since $A$ is a one dimensional space, by Corollary 2.3(2), $A$ is simultaneously approximatively compact. Note that $B$ consists of norm-divergent sequences, Then $B$ is also simultaneously approximatively compact. For any $n>m$,
we have
\[d(span\{x_0\}+x_n, span\{x_0\}+x_m)=d(x_n,x_m+span\{x_0\})>1,\]
this implies that $A+B$ is closed.
Let us now show that the sum  $A+B$ is not simultaneously proximinal. Taking $D=\{0\}$ to be a singleton set, for every $n\in \mathbb{N}$, we have
\[d(D,span\{x_0\}+x_n)=d(x_n, span\{x_0\})>1.\]
Therefore, $\|x\|>1$ for every $x\in A+B$, and $d(D,A+B)\geq1$. Note that
\[\{x_n\}\subset A+B\;\; \text{and}\;\;\;\|x_n\|\rightarrow 1,\]
then $d(D,A+B)=1$ and $P_{A+B}(D)=\emptyset$. So  $A+B$ is not simultaneously proximinal.
\end{proof}

In the following, we will discuss the preserving properties of $\tau$-compact sets and simultaneously proximinal sets under sum operation. We show that the sum of a compact convex set $C$ and a simultaneously $\tau$-strongly proximinal set (resp. simultaneously approximatively $\tau$-compact set) $D$ of a Banach space is also simultaneously $\tau$-strongly proximinal (resp. simultaneously approximatively $\tau$-compact) This implies that  $C+D$ is closed.

\begin{theorem} Let $C$ and $D$ be two convex subsets of a Banach space $X$. Assume that $C$ is compact and $D$ is simultaneously $\tau$-strongly proximinal.  Then $C+D$ is simultaneously $\tau$-strongly proximinal.
\end{theorem}
\begin{proof}
Firstly, we show that $C+D$ is simultaneously proximinal. Let $A$ be a bounded subset of $X$. Then there exist a sequence
$\{c_n\}\subset C$ and a sequence $\{d_n\}\subset D$ such that
\[\sup\limits_{a\in A}\|a-c_n-d_n\|\rightarrow d(A,C+D).\]
Since $C$ is compact, $\{c_n\}$ has a convergent subsequence $\{c_{n_k}\}$ in the norm topology. We still denote the subsequence $\{c_{n_k}\}$ as $\{c_n\}$, and let $c_n\rightarrow c$. Note that
\[\sup\limits_{a\in A}\|a-c-d_n\|\leq \sup\limits_{a\in A}\|a-c_n-d_n\|+\|c_n-c\|,\]
This implies that
\[d(A,c+D)\leq\sup\limits_{a\in A}\|a-c-d_n\|\rightarrow d(A,C+D)\leq d(A,c+D)=d(A-c,D).\]
Therefore, $d(A,C+D)=d(A-c,D)$. Since $D$ is simultaneously proximinal, there exist $d\in D$ such that
\[\sup\limits_{a\in A}\|a-c-d\|=d(A-c,D)=d(A,C+D),\]
so  $c+d\in P_{C+D}(A)$. Since $A$ is arbitrary, $C+D$ is simultaneously proximinal.

Note that $C+D$ is convex, then $C+D$ is $\tau$-closed by the simultaneous proximinality. If $C+D$ is not simultaneously $\tau$-strongly proximinal, then there exist a bounded set $A\subset X$, a $\tau$-neighbourhood $V$ of $0$ and a minimizing sequence $\{c_n+d_n\}\subset C+D$ for $A$ with $c_n+d_n\notin P_{C+D}(A)+V$ for all $n\geq 1$.
Without lose of generality, let $c_n\rightarrow c$. Thus,
\[\sup\limits_{a\in A}\|a-c-d_n\|\rightarrow d(A,C+D)=d(A,c+D)=d(A-c,D).\]
Suppose that $d\in P_{D}(A-c)$, then $\sup\limits_{a\in A}\|a-c-d\|=d(A-c,D)=d(A,C+D)$. This implies $c+d\in P_{C+D}(A)$ and
\[c+P_{D}(A-c)\subset P_{C+D}(A).\]
By the continuity of addition, there exist a $\tau$-neighbourhood $V_1$ of $0$ with $V_1+V_1\subset V$. Since $\{d_n\}\subset D$ is a minimizing sequence for $A-c$ and $D$ is simultaneously $\tau$-strongly proximinal, there is a $n_0\in \mathbb{N}$ such that for all $n\geq n_0$, $d_n\in P_D(A-c)+V_1$. Note that $c_n\rightarrow c$, there is a $n_1\in \mathbb{N}$ such that for all $n\geq n_1$, $c_n\in c+V_1$. Thus, for $n\geq \max\{n_0,n_1\}$, we have
\[c_n+d_n\in c+V_1+P_D(A-c)+V_1\subset P_{C+D}(A)+V.\]
This is a contradiction!
\end{proof}

\begin{theorem} Let $C$ and $D$ be two convex subsets of a Banach space $X$. Assume that $C$ is compact and $D$ is simultaneously approximatively $\tau$-compact. Then $C+D$ is simultaneously approximatively $\tau$-compact.
\end{theorem}
\begin{proof}
 Suppose that $C+D$ is not simultaneously approximatively $\tau$-compact for some bounded set $A\subset X$. Then there is a minimizing sequence $\{c_n+d_n\}\subset C+D$ for $A$  such that no subsequence is $\tau$-convergent. By the compactness of $C$, $\{d_n\}$ has no $\tau$-convergent subsequence. Without lose of generality, let $c_n\rightarrow c$ in norm topology. Thus,
\[\sup\limits_{a\in A}\|a-c-d_n\|\rightarrow d(A,C+D)=d(A,c+D)=d(A-c,D).\]
Therefore, $\{d_n\}\subset D$ is a minimizing sequence for $A-c$. This contradicts to the simultaneous approximatively $\tau$-compactness for $D$.
\end{proof}

\begin{remark}
Note that the sum of a weakly compact convex subset and a simultaneously approximatively compact subset may be not simultaneously approximatively compact. Let $C$ be the weakly compact subset in Example 2.5 and $D=\{0\}$. Since $D$ is a
singleton, $D$ is simultaneously approximatively compact. But $C+D=C$ is not simultaneously approximatively compact.
\end{remark}

We will show that the sum of a weakly compact convex subset $C$ and a simultaneously approximatively weakly compact convex subset $D$ is also simultaneously approximatively weakly compact. This deduces that  $C+D$ is closed.  We recall first the following useful results.

\begin{lemma}\cite{CCL}
Let $X$ be a Banach space and $C$ be a closed convex set of $X$. Then the following statements are equivalent.
~~\\
(1) $C$ is weakly compact.~~\\
(2) for every sequence $\{x_n\}\subset C$, there is convergent sequence $\{y_n\}$ satisfying $y_n\in co\{x_j\}_{j\geq n}$ \;for all $n\in\mathbb{N}$.~~\\
(3) for every sequence $\{x_n\}\subset C$, there is weakly convergent sequence $\{y_n\}$ satisfying $y_n\in co\{x_j\}_{j\geq n}$ \;for all $n\in\mathbb{N}$.~~\\
\end{lemma}

\begin{lemma}\cite{J}
Let $X$ be a Banach space and $C$ be a bounded subset of $X$. Then the following statements are equivalent.
~~\\
(1) $C$ is not relatively weakly compact.
~~\\
(2) there exists a sequence $\{y_n\}\subset C$ satisfying the James condition, i.e., there exists some $\theta>0$, such that
\[d(co(y_1,y_2,\cdots,y_k), co(y_{k+1},y_{k+2},\cdots))\geq \theta, \forall k\in\mathbb{N}.\]
\end{lemma}

\begin{lemma}
Let $X$ be a Banach space and $\{x_n\}\subset X$ be a bounded sequence. If $\{x_n\}$ has no weakly convergent subsequence. Then exists a subsequence of $\{x_n\}$ satisfying the James condition.
\end{lemma}
\begin{proof}
Since $\{x_n\}$ has no weakly convergent subsequence, $C\equiv\{x_n\}$ is not relatively weakly compact. Then there exists
a sequence $\{y_n\}\subset C$ such that $\{y_n\}$ satisfying the James condition and $\{y_n\}$ are different from each other. Therefore, $\{y_n\}$ has a subsequence $\{y_{n_k}\}$ such that $\{y_{n_k}\}$ is also a subsequence of $\{x_n\}$.
Obviously, $\{y_{n_k}\}$ satisfies the James condition, and the proof is complete.
\end{proof}

\begin{theorem} Let $C$ and $D$ be two convex subsets of a Banach space $X$. Assume that $C$ is weakly compact and $D$ is simultaneously approximatively weakly compact. Then $C+D$ is simultaneously approximatively weakly compact.
\end{theorem}

\begin{proof}
Let $A$ be a bounded subset of $X$. Then there exist a sequence $\{c_n\}\subset C$ and a sequence $\{d_n\}\subset D$ such that
\[\sup\limits_{a\in A}\|a-c_n-d_n\|\rightarrow d(A,C+D).\]
It is sufficient to show $\{c_n+d_n\}$ has a weakly convergent subsequence. By the weak compactness of $C$, it is equivalent to show that $\{d_n\}$ has a weakly convergent subsequence. Without lose of generality, we assume that
$\sup\limits_{a\in A}\|a-c_n-d_n\|\searrow d(A,C+D)$.

Conversely, suppose $\{d_n\}$ has no weakly convergent subsequence, by Lemma 3.9, there exists a subsequence $\{d_{n_k}\}$( we still denote the subsequence $\{d_{n_k}\}$ as $\{d_n\}$) satisfying the James condition. According to  Lemma 3.7, there exists a convergent sequence $y_n\in co\{c_j:j\geq n\}$, i.e., for every $n\in\mathbb{N}$, there exists $\{\lambda_{n,j}\geq0:j\geq n\}$ with $\sum\limits_{j\geq n}\lambda_{n,j}=1$ such that
 $y_n=\sum\limits_{j\geq n}\lambda_{n,j}c_j$, where $\{\lambda_{n,j}>0:j\geq n\}$ is a finite set.
Let $z_n=\sum\limits_{j\geq n}\lambda_{n,j}d_j$, then
\begin{eqnarray}\nonumber
d(A,C+D)\leq\sup\limits_{a\in A}\|a-y_n-z_n\|=\sup\limits_{a\in A}\|\sum\limits_{j\geq n}\lambda_{n,j}(a-c_j-d_j)\|\\\nonumber
\leq\sup\limits_{a\in A}\sum\limits_{j\geq n}\lambda_{n,j}\|a-c_j-d_j\|\leq\sup\limits_{a\in A}\|a-c_n-d_n\|\\\nonumber
\rightarrow d(A,C+D).
\end{eqnarray}
Let $y_n\rightarrow y$, then
\[\sup\limits_{a\in A}\|a-y-z_n\|\leq\sup\limits_{a\in A}\|a-y_n-z_n\|+\|y_n-y\|.\]
This implies that
\[d(A,y+D)\leq\sup\limits_{a\in A}\|a-y-z_n\|\rightarrow d(A,C+D)\leq d(A,y+D).\]
Therefore, $\{z_n\}\subset D$ is a minimizing sequence for $A-y$. By the simultaneously approximatively weak compactness of $D$, there exists a weakly convergent subsequence $\{z_{n_k}\}$ of $\{z_n\}$ satisfying
\[z_{n_k}\in co\{d_j: n_k\leq j<n_{k+1}\},\;\; \text{for all}~ k\in\mathbb{N}.\]
Thus, for all $k\in\mathbb{N}$
\[co(z_{n_1},z_{n_2},\cdots,z_{n_k})\subset co(d_1,d_2,\cdots,d_{n_{k+1}-1}),\]
\[co(z_{n_{k+1}},z_{n_{k+2}},\cdots)\subset co(d_{n_{k+1}},d_{n_{k+1}+1},\cdots).\]
Note that $\{d_n\}$ satisfys the James condition, then $\{z_{n_k}\}$ satisfies the James condition.
By Lemma 3.8, $\{z_{n_k}\}$ is not relatively weakly compact, this contradicts to $\{z_{n_k}\}$ is weakly convergent.
\end{proof}

\section{Sum of simultaneously proximinal subspaces}

In this section, we shall discuss the sum of simultaneously proximinal subspaces. The Lemma 4.1 below is easy to prove, see also\cite[Theorem 5.20]{RU}.
\begin{lemma} Suppose that $E$, $F$  are two closed subspace of  a Banach space $X$ satisfying $E+F$ is closed. Then there
exists $m>0$ such that $B_{E+F}\subset m(B_E+B_F)$.
\end{lemma}

\begin{theorem} Let $E$ and $F$ be two subspaces of a Banach space $X$. Assume that $E$ is a finite dimensional subspace and $F$ is a simultaneously $\tau$-strongly proximinal subspace.  Then $E+F$ is simultaneously $\tau$-strongly proximinal.
\end{theorem}

\begin{proof}
Since $E$ is finite dimensional and $F$ is $\tau$-closed, $E+F$ is closed. According to Lemma 4.1, there exists $m>0$ such that
\[B_{E+F}\subset m(B_E+B_F)\subset mB_E+F.\]
Therefore, for a bounded subset $A\subset X$, there exists a $\lambda>d(A,E+F)+\sup\limits_{a\in A}\|a\|+1$ such that
\[d(A,E+F)=d(A,\lambda(mB_E+F))=d(A,\lambda mB_E+F).\]
Since $\lambda mB_E$ is compact and $F$ is simultaneously $\tau$-strongly proximinal, by Theorem 3.4, $\lambda mB_E+F$ is simultaneously $\tau$-strongly proximinal. Thus, for any $\tau$-neighbourhood $V$ of $0$, there exists a $0<\delta<1$, such that
\[P_{\lambda mB_E+F}(A,\delta)\subset P_{\lambda mB_E+F}(A)+V.\]
Note that
\[P_{\lambda mB_E+F}(A,\delta)=P_{E+F}(A,\delta),\;\;\;P_{\lambda mB_E+F}(A)=P_{E+F}(A).\]
Thus,
\[P_{E+F}(A,\delta)\subset P_{E+F}(A)+V.\]
By the arbitrariness of $A$, $E+F$ is simultaneously $\tau$-strongly proximinal in $X$.
\end{proof}

\begin{theorem} Let $E$ and $F$ be two subspaces of a Banach space $X$. Assume that $E$ is finite dimensional and $F$ is simultaneously approximatively $\tau$-compact. Then $E+F$ is simultaneously approximatively $\tau$-compact.
\end{theorem}

\begin{proof}
Let $A$ be a bounded subset of $X$. suppose $\{e_n\}\subset E$ and  $\{f_n\}\subset F$ such that
\[\sup\limits_{a\in A}\|a-e_n-f_n\|\rightarrow d(A,E+F).\]
Thus, $\{e_n+f_n\}$ is a bounded sequence in $E+F$, and $\{e_n+f_n\}\subset kB_{E+F}$ for some $k>0$.  Since $E$ is finite dimensional and $F$ is closed, $E+F$ is closed.
According to Lemma 4.1, there exists $m>0$ such that
\[\{e_n+f_n\}\subset kB_{E+F}\subset km(B_E+B_F)\subset kmB_E+F.\]
Therefore,
\[\sup\limits_{a\in A}\|a-e_n-f_n\|\rightarrow d(A,E+F)=d(A,kmB_E+F).\]
Since $kmB_E$ is compact and $F$ is simultaneously approximatively  $\tau$-compact, by Theorem 3.5, $kmB_E+F$ is simultaneously approximatively $\tau$-compact. Then $\{e_n+f_n\}$ has a $\tau$-convergent subsequence and the proof is complete.
\end{proof}

\begin{theorem} Let $E$ and $F$ be two subspaces of a Banach space $X$. Assume that $E$ is reflexive and $F$ is simultaneously approximatively weakly compact, satisfying $E+F$ is closed. Then $E+F$ is simultaneously approximatively weakly compact.
\end{theorem}

\begin{proof}
The proof is similar to Theorem 4.3, it is sufficient to substitute weak compactness for compactness and substitute simultaneously approximatively weak compactness for simultaneously approximatively compactness.

\end{proof}

{\bf Acknowledgements.} The authors would like to thank the referee for the comments and advice which made this article better and Functional
Analysis group of Xiamen university for their very helpful conversations and suggestions.

\bibliographystyle{amsalpha}

\end{document}